\documentclass{amsart}

\usepackage{amssymb,amsmath}
\usepackage[section]{algorithm}
\usepackage{algpseudocode}

\newtheorem{theorem}{Theorem}[section]
\newtheorem{proposition}[theorem]{Proposition}
\newtheorem{lemma}[theorem]{Lemma}
\newtheorem{corollary}[theorem]{Corollary}

\newtheorem{definition}[theorem]{Definition}

\newtheorem{remark}[theorem]{Remark}

\newtheorem{example}[theorem]{Example}

\def\N{{\mathbb N}}
\def\Z{{\mathbb Z}}

\def\st{\hbox{s.t.}}

\def\lcm{{\mathrm{lcm}}}

\def\deg{{\mathrm{deg}}}

\def\hN{{\hat{\N}}}

\def\bX{{\bar{X}}}
\def\bF{{\bar{F}}}
\def\bP{{\bar{P}}}

\def\biota{{\bar{\iota}}}
\def\bV{{\bar{V}}}
\def\bL{{\bar{L}}}

\def\w{{\mathrm{w}}}

\def\Im{{\mathrm{Im\,}}}
\def\End{{\mathrm{End}}}
\def\Mon{{\mathrm{Mon}}}
\def\Sat{{\mathrm{Sat}}}

\def\Gr{Gr\"obner}

\def\lm{{\mathrm{lm}}}
\def\lc{{\mathrm{lc}}}
\def\lt{{\mathrm{lt}}}
\def\LM{{\mathrm{LM}}}
\def\spoly{{\mathrm{spoly}}}

\def\Reduce{{\textsc{Reduce}}}
\def\HFreeGBasis{{\textsc{HFreeGBasis}}}
\def\FreeGBasis{{\textsc{FreeGBasis}}}
\def\SigmaGBasis{{\textsc{SigmaGBasis}}}

\begin{document}

\title[Extended letterplace correspondence]
{Extended letterplace correspondence for nongraded noncommutative ideals
and related algorithms}

\author[R. La Scala]{Roberto La Scala$^*$}

\address{$^*$ Dipartimento di Matematica, Universit\`a di Bari, Via Orabona 4,
70125 Bari, Italy}
\email{roberto.lascala@uniba.it}

\thanks{Partially supported by Universit\`a di Bari}



\maketitle

\begin{abstract}
Let $K\langle x_i\rangle$ be the free associative algebra generated
by a finite or a countable number of variables $x_i$. The notion of
``letterplace correspondence'' introduced in \cite{LSL1,LSL2} for
the graded (two-sided) ideals of $K \langle x_i \rangle$ is extended
in this paper also to the nongraded case. This amounts to the possibility
of modelizing nongraded noncommutative presented algebras by means
of a class of graded commutative algebras that are invariant under
the action of the monoid $\N$ of natural numbers. For such purpose
we develop the notion of saturation for the graded ideals of
$K\langle x_i,t\rangle$, where $t$ is an extra variable and for their
letterplace analogues in the commutative polynomial algebra $K[x_{ij},t_j]$,
where $j$ ranges in $\N$. In particular, one obtains an alternative algorithm
for computing inhomogeneous noncommutative \Gr\ bases using just homogeneous
commutative polynomials. The feasibility of the proposed methods is shown
by an experimental implementation developed in the computer algebra system
Maple and by using standard routines for the Buchberger algorithm contained
in Singular.
\end{abstract}



\section{Introduction}

Many structures and models in mathematics and physics are based on
noncommutative associative algebras that are given by a presentation
with a finite or a countable number of generators. It is sufficient
to mention the role of Hecke algebras or Temperley-Lieb ones
in statistical mechanics and noncommutative geometry \cite{Cn,KL},
as well as the relevance of more classical enveloping algebras \cite{Dx}
or relatively free algebras defined for PI-algebras \cite{Dr,GZ}.
A systematic way to control the consequences of the defining relations
of a presented algebra consists in considering a well-ordering on the
monomials of the free associative algebra (tensor algebra) which
is compatible with multiplication and in computing what is modernly
called a ``\Gr\ basis'' or a ``\Gr-Shirshov basis''. In fact,
if it is possible to describe such a basis for the two-sided ideal
of the relations satisfied by the generators of the associative algebra
then a monomial linear basis is given for it, that is, one has some
kind of generalization of the Poincar\'e-Birkhoff-Witt theorem.

Among the founding contributions to the theory of \Gr\ bases
for associative algebras one has to mention of course Bruno Buchberger
\cite{Bu} for the commutative case and the fundamental papers
\cite{Be, Gr, Mo, Uf1, Uf2} for the non-commutative one. For nonassociative
algebras one finds the roots of this algorithmic theory in the pioneeristic
work of Anatolii Shirshov \cite{S1,S2}. To explore some history and
the wide range of applications of the modern theory of noncommutative
\Gr\ bases one can see, for instance, \cite{BC,BCS,BK}.

Starting with the papers \cite{LSL1,LSL2}, through a substantial
development of the concept of letterplace embedding contained in \cite{DRS},
a new approach for the theory and computation of noncommutative \Gr\ bases
has been proposed. The basic idea is to define a bijective
correspondence between all graded two-sided ideals of the free associative
algebra and a class of multigraded invariant ideals of a commutative polynomial
algebra in double-indexed (letter-place) variables where shift operators
act over the place indices. Such bijection provides also a correspondence
between the homogeneous \Gr\ bases of these ideals. It follows that
the notion of \Gr\ basis in the commutative and noncommutative case
and the related algorithms can be considered as special instances
of a general theory of \Gr\ bases for commutative ideals that are
invariant under the action of suitable algebra endomorphisms
\cite{BD,GLS,LSL2,LS}. Since the endomorphisms acting on the
letterplace algebra are just shift operators, note that these results
contribute also to the theory of algebras of finite difference
polynomials \cite{Co,Le}.

The goal of the present paper is to complete the work initiated in 
\cite{LSL1,LSL2} by proposing an extension of the letterplace correspondence
to the nongraded case. This is obtained by analyzing in detail the concept
of saturation for nongraded ideals of the free associative algebra
and for their letterplace analogues. Note that the homogenization and
saturation processes for the noncommutative case were previously introduced
in \cite{LiSu,No,Uf3} (see also \cite{Li}). From the extended letterplace
correspondence one obtains an alternative algorithm to compute inhomogeneous
noncommutative \Gr\ bases by using homogeneous polynomials in commutative
variables.
In fact, these methods can be easily implemented in any commutative
computer algebra system. Then, one has that the theory and methods
for commutative and noncommutative \Gr\ bases are unified whenever
they are homogeneous or not. The feasibily of the proposed algorithms
is shown in practice by means of an experimental implementation and
a test set consisting of relevant classes of noncommutative algebras. 

In Section 2 we describe the bijective correspondence between all
(two-sided) ideals of the free associative algebra $F = K \langle X \rangle$
and the class of saturated graded ideals of the algebra $\bF =
K \langle \bX \rangle$, where $\bX = X\cup\{t\}$. If $\N = \{n\in\Z\mid
n\geq 0\}$ and $\N^* = \N\setminus\{0\}$, for the letterplace
algebras $P = K[X\times\N^*]$ and $\bP = K[\bX\times\N^*]$ we introduce
the action of the monoid $(\N,+)$ on the place indices of the variables
and also a multigrading based on such indices. Then, one obtains a bijection
between all $\N$-invariant ideals of $P$ and the class of saturated
multigraded $\N$-ideals of $\bP$. In Section 3 we review some key results
proved in \cite{LSL1,LSL2}. Precisely, the {\em letterplace ideals} of $P$
are defined as $\N$-ideals generated by elements that are multilinear
with respect to the place multigrading. Then, we introduce the {\em letterplace
correspondence} as a bijection between all graded ideals of $F$ and
the class of letterplace ideals of $P$. Note that under this correspondence
a saturated ideal of $\bF$ does not map into a saturated ideal of $\bP$.
It is necessary therefore to introduce the notion of {\em $L$-saturation}
for letterplace ideals as a saturation property that involves only multilinear
elements. By composing the above ideal correspondences, we finally obtain
the {\em extended letterplace correspondence} which maps all ideals of $F$
into the class of $L$-saturated letterplace ideals of $\bP$.

To develop effective methods for the $L$-saturation, in Section 4
we review the notion of monomial $\N$-ordering of $P$ and the constructin
of an important class of such orderings that we call {\em place $\N$-orderings}.
Then, we prove that they induce the graded right lexicographic ordering
of the free associative algebra $F$. We review finally the theory
of \Gr\ $\N$-bases for ideals of $P$ that are invariant under shift operators
and the related {\em letterplace algorithm} that computes homogeneous
noncommutative \Gr\ bases by using just elements of the commutative
algebra $P$. In Section 5 we solve the problem of computing $L$-saturations
of letterplace ideals by using {\em \Gr\ $L$-bases} that are \Gr\ $\N$-bases
restricted to multilinear elements. The monomial orderings of $\bP$
suitable for this task are place $\N$-orderings which are of elimination
for the extra variables $t(j)$. As a byproduct one obtains finally
a letterplace algorithm for computing inhomogeneous noncommutative
\Gr\ bases using homogeneous polynomials of $\bP$. This method is
illustrated in a detailed simple example in Section 6 and it is experimented
in Section 7 for classes of presented associative algebras that are
of interest in different areas of algebra. The experiments
are performed by means of an implementation developed in the language
of Maple and also by using standard routines for the Buchberger algorithm
that are implemented in \textsf{Singular} \cite{DGPS}. Conclusions
about the letterplace approach to noncommutative computations
and further developments of it are finally discussed in Section 8.


\section{Homogenized and saturated ideals}

We start studying the notion of homogenization and saturation for ideals
of the free associative algebra. These concepts have been introduced
essentially in \cite{LiSu,No,Uf3} but we intend to clarify why commutators
naturally arise in such constructions.
Denote by $F = K\langle X \rangle$ the free associative algebra
freely generated by a finite or a countable set $X = \{x_1,x_2,\ldots\}$.
Clearly, one has the algebra grading $F = \bigoplus_{d\in\N} F_d$ where
$F_d$ is the subspace of homogeneous polynomials of total degree $d$.
Let $t$ be a new variable disjoint by $X$. Define $\bX = X\cup\{t\},
\bF = K\langle \bX \rangle$. Consider the algebra endomorphism
$\varphi:\bF\to\bF$ such that $x_i\mapsto x_i$ and $t\mapsto 1$
for all $i\geq 1$. Clearly $\varphi^2 = \varphi$ and $F = \varphi(\bF)$.
Then, the map $\varphi$ defines a bijective correspondence between all
two-sided ideals of $F$ and two-sided ideals of $\bF$ containing
$\ker\varphi = \langle t - 1 \rangle$. In what follows, all the ideals
of the algebras $F,\bF$ are {\em assumed} two-sided ones.

\begin{definition}
Denote by $C$ the largest graded ideal contained in $\ker\varphi$, that is,
the ideal $C$ is generated by all homogeneous elements $f\in\bF$ such that
$\varphi(f) = 0$.
\end{definition}

\begin{proposition}
The ideal $C\subset\bF$ is generated by the commutators $[x_i,t] =
x_i t - t x_i$, for any $i\geq 1$.
\end{proposition}

\begin{proof}
Let $f\in\bF$ be a homogeneous element such that $\varphi(f) = 0$.
Since the commutators $[x_i,t]$ clearly belongs to $C$, we have
to prove that $f$ is congruent to $0$ modulo them. In fact, it is clear
that $f$ is congruent to a homogeneous element $f' = t^{d'} \sum_k f_k t^{d-k}$
where $d'\geq 0$ and $f_k\in F$ is homogeneous of degree $k$, for any $k$.
Then $0 = \varphi(f) = \varphi(f') = \sum_k f_k$ and hence $f_k = 0$
for all $k$. We conclude that $f' = 0$.
\end{proof}

We want now to define a bijective correspondence between all ideals of $F$
and some class of graded ideals of $\bF$ containing $C$.

\begin{definition}
Let $I$ be any ideal of $F$. We define $I^*\subset\bF$ to be the largest
graded ideal contained in the preimage $\varphi^{-1}(I)$, that is,
the ideal $I^*$ is generated by all homogeneous elements
in $\varphi^{-1}(I)$. Clearly $C = 0^*\subset I^*$. We call $I^*$
the {\em homogenization} of the ideal $I$.
\end{definition}

\begin{definition}
Let $f\in F, f\neq 0$ and denote by $f = \sum_k f_k$ the decomposition
of $f$ in its homogeneous components. We denote $\deg(f) =
d = \max\{k\}$ and define $f^* = \sum_k f_k t^{d-k}$. We call $\deg(f)$
the {\em top degree} of $f$ and $f^*$ its {\em homogenization}.
Clearly $f^*\in\bF$ is a homogeneous element such that $\deg(f^*) = 
\deg(f)$ and $\varphi(f^*) = f$.
\end{definition}

\begin{theorem}
Let $I$ be an ideal of $F$. Then $I^* = \langle f^*\mid f\in I,
f\neq 0\rangle + C$.
\end{theorem}

\begin{proof}
Denote $J = \langle f^*\mid f\in I,f\neq 0\rangle + C$. 
Clearly $J$ is a graded ideal of $\bF$ such that $\varphi(J)\subset I$ and
hence $J\subset I^*$. Let $g\in I^*$ be a homogeneous element and define
$f = \varphi(g)\in I$. If $f = 0$ then $g\in C\subset J$. Otherwise,
denote $d = \deg(f)$ and $d' = \deg(g)$. Since clearly $d'\geq d$ one has
that $g$ is congruent modulo $C$ to the element $t^{d'-d} f^*$ and
hence $g\in J$.
\end{proof}

If $I\subset F$ is an ideal one has clearly that $\varphi(I^*) = I$.
Moreover, if $J\subset\bF$ is a graded ideal containing $C$
then in general $J\subset \varphi(J)^*$.

\begin{definition}
Let $J\subset\bF$ be a graded ideal which contains $C$. Define
$\Sat(J) = \varphi(J)^* =
\langle \varphi(f)^* \mid f\in J,f\notin C,f\ \mbox{homogeneous} \rangle + C$.
Then $\Sat(J)\subset\bF$ is a graded ideal containing $J$ that we call
the {\em saturation} of $J$.
\end{definition}

\begin{definition}
Let $J\subset\bF$ be a graded ideal containing $C$.
We say that $J$ is {\em saturated} if $J$ coincides with its saturation
$\Sat(J)$, that is, for any homogeneous element $f\in J,f\notin C$
one has that $\varphi(f)^*\in J$. If $I$ is an ideal of $F$ then its
homogenization $I^*$ is clearly a saturated ideal.
\end{definition}

Note that in \cite{LiSu} an equivalent definition of saturated ideal
is named {\em dh-closed}. Then, a bijective correspondence is given
between all ideals of $F$ and the saturated graded ideals of $\bF$
containing $C$. One can characterize such ideals in the following way.

\begin{theorem}
\label{satchar}
Let $J\subset\bF$ be a graded ideal containing $C$. Then $J$ is saturated
if and only if $t f\in J$ with $f\in\bF$ implies that $f\in J$.
\end{theorem}

\begin{proof}
Suppose that $J$ is saturated and let $t g\in J$ with $g\in\bF$. Since $J$
is graded, we can assume that $g$ is homogeneous. Put $f = \varphi(g) =
\varphi(t g)$. If $f = 0$ then $g\in C\subset J$. Otherwise, since $J$
is saturated and $t g\in J$ we obtain that $f^*\in J$. Moreover, one has
clearly that $g$ is congruent modulo $C$ to an element $t^d f^*\in J$
for some $d\geq 0$ and hence $g\in J$.
Suppose now that $t g\in J$ implies $g\in J$ and let $g\in J,g\notin C$
be a homogeneous element. If $f = \varphi(g)$ then $g$ is congruent
modulo $C\subset J$ to an element $t^d f^*$. We conclude that $t^d f^*\in J$
and therefore $f^*\in J$.
\end{proof}

\begin{corollary}
Let $J\subset\bF$ be a graded ideal containing $C$. Then, we have that
$\Sat(J) = \{f\in\bF\mid t^i f\in J,\ \mbox{for some}\ i\geq 0\}$.
\end{corollary}

\begin{proof}
Denote $J' = \{f\mid t^i f\in J,\ \mbox{for some}\ i\}$.
Let $g\in\bF$ and $f\in J'$, that is, $t^i f\in J$ for some $i$.
Clearly $g t^i f\in J$ and also $t^i g f\in J$ since $C\subset J$.
We conclude that $g f\in J'$. With similar arguments one proves
that $J'$ is a graded ideal of $\bF$ containing $J$.
Moreover, by Theorem \ref{satchar} it follows immediately that
$J'$ is a saturated ideal. Finally, we have clearly that $\varphi(J') =
\varphi(J)$ and hence $J' = \varphi(J')^* = \varphi(J)^* = \Sat(J)$.
\end{proof}


We start now considering commutative polynomial algebras
with the purpose of defining analogues of the above
noncommutative constructions. Denote $\N^* = \N\setminus\{0\}$
and consider the product set $X(\N^*) = X\times\N^*$. For the
elements of this set we make use of the notation $x_i(j) = (x_i,j)$,
for all $i,j\geq 1$. Define $P = K[X(\N^*)]$ the polynomial algebra
in all commuting variables $x_i(j)$. The algebra $P$ is called
the {\em letterplace algebra} \cite{DRS}. Denote by $\End(P)$
the monoid of all algebra endomorphisms of $P$. A monoid homomorphism
$\rho:\N\to\End(P)$ is defined by putting $\rho(k)(x_i(j)) = x_i(k+j)$,
for all $k\geq 0$ and for any $i,j\geq 1$. We say therefore that
$P$ is an {\em $\N$-algebra}. In fact, $P$ is a free $\N$-algebra
generated by the set $X(1) = \{x_i(1)\mid i\geq 1\}$ (see \cite{LS,GLS}).
We make use of the notation $k\cdot f = \rho(k)(f)$, for all $k\geq 0$
and for any $f\in P$. Note finally that $\rho$ together with all $\rho(k)$
are injective maps.
An ideal $I\subset P$ is called an {\em $\N$-invariant ideal} or
an {\em $\N$-ideal} if $\N\cdot I\subset I$.
Clearly, we have the algebra grading $P = \bigoplus_{d\in\N} P_d$ where
$P_d$ is the subspace of homogeneous polynomials of total degree $d$.
The algebra $P$ has another natural multigrading defined as follows.
If $m = x_{i_1}(j_1)\cdots x_{i_d}(j_d)\in\Mon(P)$ then we denote
$\partial(m) = \mu = (\mu_k)_{k\in\N^*}$ where $\mu_k =
\#\{\alpha \mid j_\alpha = k\}$. If $P_\mu\subset P$ is the subspace
spanned by all monomials of multidegree $\mu$ then $P = \bigoplus_\mu P_\mu$
is clearly a multigrading. Note that the multidegrees $\mu = (\mu_k)_{k\in\N^*}$
have finite support and one can define $|\mu| = \sum_k \mu_k$. 
Then, one has clearly that $P_d = \bigoplus_{|\mu|=d} P_\mu$, that is,
the multihomogeneous elements are also homogeneous ones. Note that
the multigrading is compatible with the $\N$-algebra structure on $P$.
Precisely, if $\mu = (\mu_k)$ is a multidegree then we denote
$i\cdot\mu = (\mu_{k-i})_{k\in\N^*}$ where we put $\mu_{k-i} = 0$
when $k-i < 1$. Then, for all $i\geq 0$ and for any multidegree $\mu$
one has that $i\cdot P_\mu\subset P_{i\cdot\mu}$.

Define $\bP = K[\bX(\N^*)]$ and consider the $\N$-algebra endomorphism
$\psi:\bP\to\bP$ such that $x_i(1)\mapsto x_i(1)$ and $t(1)\mapsto 1$
for all $i\geq 1$. Clearly, the map $\psi$ is idempotent and
$P = \psi(\bP)$. Moreover, one has that the $\N$-ideal $\ker(\psi) =
\langle t(1) - 1 \rangle_\N$ does not contain any multihomogeneous
element different from zero. We define now a bijective correspondence
between all $\N$-ideals of $P$ and some class of multigraded $\N$-ideals
of $\bP$.

\begin{definition}
Let $I$ be any $\N$-ideal of $P$. We define $I^*\subset\bP$ to be the largest
multigraded $\N$-ideal contained in the preimage $\psi^{-1}(I)$, that is,
the ideal $I^*$ is generated by all multihomogeneous elements in $\psi^{-1}(I)$.
We call $I^*$ the {\em multihomogenization} of the ideal $I$.
Note that $0^* = 0$.
\end{definition}

\begin{definition}
Let $f\in P, f\neq 0$ and denote by $f = \sum_\mu f_\mu$ the decomposition
of $f$ in its multihomogeneous components. We denote $\partial(f) =
\nu = (\max_\mu\{\mu_k\})_{k\in\N^*}$ and define $f^* = \sum_\mu f_\mu
\prod_k t(k)^{\nu_k-\mu_k}$. We call $\partial(f)$ the {\em top multidegree}
of $f$ and $f^*$ its {\em multihomogenization}. Clearly $f^*\in\bP$ is a
multihomogeneous element such that $\partial(f^*) = \partial(f)$
and $\psi(f^*) = f$. Moreover, one has that $(i\cdot f)^* =
i\cdot f^*$, for all $i\geq 0$.
\end{definition}

\begin{theorem}
Let $I$ be an $\N$-ideal of $P$. Then $I^* = \langle f^*\mid f\in I,
f\neq 0\rangle$.
\end{theorem}

\begin{proof}
Denote $J = \langle f^*\mid f\in I,f\neq 0\rangle$. Clearly $J$ is a
multigraded $\N$-ideal of $\bP$ such that $\psi(J)\subset I$ and hence
$J\subset I^*$. Let $g\in I^*$ be a multihomogeneous element
and define $f = \psi(g)\in I$. Denote $\mu = \partial(f)$ and $\nu =
\partial(g)$. Since clearly $\nu_k\geq \mu_k$ for all $k$, one has
that $g = \prod_k t(k)^{\nu_k - \mu_k} f^*$ and hence $g\in J$.
\end{proof}

If $I\subset P$ is an $\N$-ideal one has clearly that $\psi(I^*) = I$.
Moreover, if $J\subset\bP$ is a multigraded $\N$-ideal then in general
$J\subset \psi(J)^*$.

\begin{definition}
Let $J\subset\bP$ be a multigraded $\N$-ideal. Define $\Sat(J) = \psi(J)^* =
\langle \psi(f)^* \mid f\in J,f\ \mbox{multihomogeneous} \rangle$.
Then $\Sat(J)\subset\bP$ is a multigraded $\N$-ideal containing $J$
that we call the {\em saturation} of $J$.
\end{definition}

\begin{definition}
Let $J\subset\bP$ be a multigraded $\N$-ideal. We say that $J$ is
{\em saturated} if $J$ coincides with its saturation $\Sat(J)$, that is,
if $f\in J$ is a multihomogeneous element then $\psi(f)^*\in J$. If $I$
is an $\N$-ideal of $P$ then its multihomogenization $I^*$ is clearly
a saturated ideal.
\end{definition}

Then, a bijective correspondence is given between all $\N$-ideals of $P$
and the saturated multigraded $\N$-ideals of $\bP$. One can characterize
such ideals in the following way.

\begin{theorem}
\label{satchar2}
Let $J\subset\bP$ be a multigraded $\N$-ideal. Then $J$ is saturated
if and only if $t(j)f\in J$ with $f\in\bP,j\geq 1$ implies that $f\in J$.
\end{theorem}

\begin{proof}
Suppose that $J$ is saturated and let $t(j) g\in J$ with $g\in\bF,j\geq 1$.
Since $J$ is multigraded, we can assume that $g$ is multihomogeneous.
Put $f = \psi(g) = \psi(t(j) g)$. Since $J$ is saturated and $t(j) g\in J$
we obtain that $f^*\in J$. Moreover, one has that $g =
\prod_k t(k)^{\mu_k} f^*\in J$ for some multidegree $\mu$ and hence $g\in J$.
Suppose now that $t(j) g\in J$ implies $g\in J$ and let $g\in J$ be a
multihomogeneous element. If $f = \psi(g)$ then clearly
$\prod_k t(k)^{\mu_k} f^* = g\in J$, for some $\mu$. We conclude that $f^*\in J$.
\end{proof}

\begin{corollary}
Let $J\subset\bP$ be a multigraded $\N$-ideal. Then, we have that $\Sat(J) =
\{f\in\bP\mid \prod_k t(k)^{\mu_k} f\in J,\ \mbox{for some multidegree}\ \mu\}$.
\end{corollary}

\begin{proof}
Put $J' = \{f\mid \prod_k t(k)^{\mu_k} f\in J,\ \mbox{for some}\ \mu\}$.
Let $i\geq 0$ and $f\in J'$, that is, $m f\in J$ for some $m =
\prod_k t(k)^{\mu_k}$. Since $J$ is an $\N$-ideal, we have that
$(i\cdot m)(i\cdot f) = i\cdot (m f)\in J$ where $(i\cdot m) =
\prod_k t(i+k)^{\mu_k}$. We conclude that $i\cdot f\in J'$.
With similar arguments one proves that $J'$ is a multigraded $\N$-ideal
of $\bP$ containing $J$. By Theorem \ref{satchar2} we obtain also
that $J'$ is a saturated ideal. Finally, we have clearly that $\psi(J') =
\psi(J)$ and hence $J' = \psi(J')^* = \psi(J)^* = \Sat(J)$.
\end{proof}


\section{Letterplace correspondence and $L$-saturation}

Consider the $K$-linear embedding $\iota:F\to P$ such that
$\iota(m) = x_{i_1}(1)\cdots x_{i_d}(d)$ for all monomials
$m = x_{i_1}\cdots x_{i_d}\in\Mon(F)$. This mapping was
introduced in \cite{DRS}. Note that the map $\iota$
preserves the total degree. Then, define $V = \bigoplus_d V_d$
the graded subspace of $P$ which is the image of map $\iota$.
For all $d\geq 0$, denote by $1^d$ the multidegree $\mu = (\mu_k)$ such that
$\mu_k = 1$ for $k\leq d$ and $\mu_k = 0$ otherwise. Clearly
one has that $V_d = P_{1^d}$.

\begin{definition}
Denote by $L = \bigcup_d V_d$ the set of multihomogeneous elements
of $V$. We call such elements the {\em multilinear elements of $P$}.
\end{definition}

There is a bijective correspondence between all graded ideals of $F$
and some class of multigraded $\N$-ideals of $P$. This class
is defined as follows.

\begin{definition}
Let $J$ be an $\N$-ideal of $P$. We call $J$ a {\em letterplace
ideal} or {\em $L$-ideal} or {\em multilinear $\N$-ideal}
if $J = \langle J\cap L \rangle_\N$, that is, $J$ is $\N$-generated
by multilinear elements. Clearly $J$ is a multigraded ideal.
\end{definition}

The following key result has been proved in \cite{LSL1}

\begin{theorem}
Let $I\subset F$ be a graded ideal and denote $J = \langle \iota(I) \rangle_\N$.
Then $J\subset P$ is an $L$-ideal. Conversely, let $J\subset P$ be
an $L$-ideal and denote $I = \iota^{-1}(J\cap V)$. Then $I\subset F$
is a graded ideal. Moreover, the mappings $I\mapsto J$ and $J\mapsto I$
define a bijective correspondence between graded ideals of $F$ and
letterplace ideals of $P$. Hence, we call $J$ the {\em letterplace analogue}
of $I$.
\end{theorem}

We assume now that the above result is extended to the algebras $\bF,\bP$.
Then, we make use of notations $\biota:\bF\to\bP, \bV = \Im\biota$ and
$\bL = \bigcup_d \bV_d$. Consider the letterplace analogue $D$ of the ideal
$C = 0^*$. In other words, we have that $D\subset\bP$ is the $\N$-ideal generated
by the multilinear elements $\biota([x_i,t]) = x_i(1)t(2) - t(1)x_i(2)$,
for all $i\geq 1$. Note that $D$ is not a saturated ideal. In fact, the ideal
$D$ contains the element $t(1) f$, but not $f = x_1(1)x_2(2) - x_2(1)x_1(2)$.
Moreover, its saturation $\Sat(D)$ is not an $L$-ideal, that is,
this ideal is not generated by multilinear elements. For instance,
the element $x_1(1)t(3) - t(1)x_1(3)\notin\bL$ is contained in $\Sat(D)$.
More generally, the letterplace analogue of a saturated ideal of $\bF$
is not saturated and its saturation is not a letterplace ideal.
This suggests that one needs a different notion of saturation for
such analogues that are in bijective correspondence with all ideals of $F$.
To motivate the following definition, note also that if $f$ and $t(j) f$
are multilinear elements then necessarily $j = \deg(f) + 1$.

\begin{definition}
Let $J\subset\bP$ be an $L$-ideal which contains $D$. We say that $J$ is
{\em $L$-saturated} or {\em multilinearly saturated} if $t(d+1) f\in J$
with $f\in\bL$ and $d = \deg(f)$ implies that $f\in J$.
\end{definition}

\begin{proposition}
Let $J\subset\bP$ be an $L$-ideal containing $D$. If we denote
$\Sat_L(J) = \langle f\in \bL\mid \prod_{d<k\leq d'} t(k) f\in J,
\ \mbox{for some}\ d'\geq d = \deg(f)\rangle_\N$ then $\Sat_L(J)$
is an $L$-saturated letterplace ideal containing $J$. We call $\Sat_L(J)$
the {\em $L$-saturation} or {\em multilinear saturation} of $J$ and
one has clearly that $\Sat_L(J)\subset\Sat(J)$.
\end{proposition}

\begin{proof}
By definition, one has that $J' = \Sat_L(J)$ is an $L$-ideal
that contains $J\supset D$. Denote $m_l = \prod_{0<k\leq l} t(k)$
and suppose $g(d\cdot m_1)\in J'$ with $g\in \bV_d$ and $d\geq 0$.
It remains to prove that $g\in J'$, that is, the ideal $J'$
is $L$-saturated. By definition of $J'$ we have that $g(d\cdot m_1) =
\sum_i f_i (d_i\cdot g_i)$ with $f_i\in\bV_{d_i}, g_i\in\bV_{d-d_i+1}$
and $f_i (d_i\cdot m_{l_i})\in J$, for some $l_i\geq 0$.
If $l = \max\{l_i\}$ then the element $g (d\cdot m_{l+1}) =
g (d\cdot m_1)((d+1)\cdot m_l)$ is congruent modulo $D\subset J$ to
$\sum_i f_i (d_i\cdot m_l) ((d_i + l)\cdot g_i)\in J$ and therefore
$g\in J'$.
\end{proof}

\begin{theorem}
\label{sat2lsat}
Let $I\subset \bF$ be an ideal containing $C$ which is saturated
and denote by $J\subset\bP$ the letterplace analogue of $I$
(hence $D\subset J$). Then $J$ is an $L$-saturated ideal.
\end{theorem}

\begin{proof}
Assume $g t(d+1)\in J$ with $g\in \bV_d$, for some $d\geq 0$.
Then, let $f\in \bF_d$ such that $\biota(f) = g$. We have that
$\biota(f t) = g t(d+1)\in J\cap \bV$, that is, $f t\in I$ and therefore
$f\in I$ since $C\subset I$ is a saturated ideal. We conclude that
$g\in J$.
\end{proof}

\begin{proposition}
Let $J\subset\bP$ be a letterplace ideal containing $D$ which is $L$-saturated
and put $I = \biota^{-1}(J\cap \bV)\subset\bF$. Clearly $C\subset I$ and
one has that $I$ is a saturated ideal.
\end{proposition}

\begin{proof}
It is sufficient to reverse the argument of Theorem \ref{sat2lsat}.
\end{proof}

We obtain therefore a bijective correspondence between all ideals
of $F$ and the class of $L$-saturated letterplace ideals of $\bP$.
We call this bijection the {\em extended letterplace correspondence}.

\begin{definition}
Let $I$ be any ideal of $F$ and denote by $J\subset \bP$ the letterplace analogue
of $I^*\subset \bF$ (hence $C\subset I^*$ and $D\subset J$). We call $J$
the {\em extended letterplace analogue} of $I$. Clearly, one has that
$J = \langle \iota(f^*)\mid f\in I,f\neq 0 \rangle_\N + D$ and
$I = \varphi\biota^{-1}(J\cap \bV)$.
\end{definition}

With the notations of the above definition, by Theorem \ref{sat2lsat}
we have that $J = \Sat_L(J)$. Then, it is natural to ask what is
the ideal $\Sat(J)$ extending $J$.

Denote by $Q = K[X(1)]$ the polynomial algebra in the variables $x_i(1)$
($x_i\in X$) and consider the natural algebra epimorphism $\eta:F\to Q$
such that $x_i\mapsto x_i(1)$, for all $i\geq 1$. Assume that $\N$ acts on $Q$
in the trivial way, that is, $j\cdot x_i(1) = x_i(1)$ for any $j\geq 0$.
Then, one has the $\N$-algebra epimorphism $\theta:P\to Q$ such that
$x_i(j)\mapsto x_i(1)$, for all $i,j\geq 1$. The kernel of $\theta$
is clearly the $\N$-ideal $E$ generated by the elements
$x_i(1) - x_i(2)$, for all $i$. Note that $E = \psi(D)$ and
hence $E^* = \Sat(D)$. 

\begin{theorem}
Let $I$ be any ideal of $F$ and put $I' = \theta^{-1}{\eta(I)}$.
Clearly $I'\subset P$ is an $\N$-ideal containing $E$. Denote by
$J\subset \bP$ the extended letterplace analogue of $I$.
Then, we have that $\Sat(J) = {I'}^*$.
\end{theorem}

\begin{proof}
Since $J$ is a multigraded $\N$-ideal of $\bP$, it is sufficient
to show that $\psi(J) = I'$. Consider any element $g'\in I'$.
Clearly $g'$ is congruent modulo $E = \ker\theta$ to an element
$\eta(f)\in Q\subset P$, for some $f\in I$. If $\eta(f) = 0$ then
$g'\in E = \psi(D)$ where $D\subset J$. Otherwise, we have that $f\neq 0$
and one can consider $f^*\in I^*$ and hence $g = \biota(f^*)\in J$.
It is clear that $\theta\psi(g) = \eta(f)$, that is, $\psi(g)$ is congruent
modulo $E$ to the element $\eta(f)$. Then, $\psi(g)$ is congruent
also to $g'$, that is, $g' = \psi(g) + h$ with $h\in E$. Since $E = \psi(D)$
and $D\subset J$, we conclude that $g'\in\psi(J)$. With similar arguments
one proves also that $\psi(J)\subset I'$.
\end{proof}

Assume now one wants to compute the extended letterplace analogue
$J\subset\bP$ of any ideal $I\subset F$. If $I$ is given by a generating set $G$,
we may form the graded ideal $I' = C + \langle f^*\mid f\in G\rangle\subset\bF$
and then its letterplace analogue $J'\subset\bP$. One has clearly that
$\Sat(I') = I^*$ and $\Sat_L(J') = J$. It is well know that for the
commutative case \cite{BCR,Ei} a standard tool to compute saturation
consists in performing \Gr\ bases with respect to appropriate monomial
orderings. Aiming to have a similar method for $L$-saturation, in the
next section we review the \Gr\ bases theory for letterplace ideals that
has been introduced in \cite{LSL1,LSL2}.


\section{\Gr\ $\N$-bases of letterplace ideals}

Since letterplace ideals are a special class of $\N$-ideals, 
a first step consists in introducing monomial orderings for the polynomial
algebra $P$ that are compatible with the action of $\N$. Owing to the
Higman's Lemma, one can provide $P = K[X(\N^*)]$ by monomial orderings
even if the set $X(\N^*)$ is infinite. For that purpose, this lemma
can be stated in the following way (see for instance \cite{AH},
Corollary 2.3).

\begin{proposition}
\label{higman}
Let $\prec$ be a total ordering on $M = \Mon(P)$ such that
\begin{itemize}
\item[(i)] $1\preceq m$ for all $m\in M$;
\item[(ii)] $\prec$ is compatible with multiplication on $M$, that is,
if $m\prec n$ then $t m\prec t n $, for any $m,n,t\in M$.
\end{itemize}
If $\prec$ induces a well-ordering on the variables set $X(\N^*)\subset M$
then $\prec$ is also a well-ordering on $M$ and hence it is a monomial
ordering of $P$.
\end{proposition}

We can easily assign well-orderings to the countable set $X(\N^*)$ which is 
in bijective correspondence to $\N^2$. Note now that the monoid $\N$ stabilizes
the variables set $X(\N^*)$ and hence the monomials set $M$. We have then
the following notion.

\begin{definition}
Let $\prec$ be a monomial ordering of $P$. We call $\prec$ a {\em (monomial)
$\N$-ordering of $P$} if $m\prec n$ implies that $i\cdot m\prec i\cdot n$,
for all $m,n\in M$ and $i\geq 0$.
\end{definition}

One defines a main class of $\N$-orderings of $P$ in the following way.
Denote $P(j) = K[x_i(j)\mid i\geq 1]$ and put $M(j) = \Mon(P(j))$.
Clearly $P = \bigotimes_{j\geq 1} P(j)$, that is, all monomials $m\in M$
can be factorized as $m = m_{j_1}\cdots m_{j_k}$, where $m_{j_s}\in M(j_s)$
and $j_1 > \ldots > j_k$.  Let $\rho:\N\to\End(P)$ be the monoid
homomorphism corresponding to the action of $\N$ over $P$.
For any $j\geq 0$, one has that the map $\rho(j)$ defines an isomorphism
between the monoids $M(1),M(j+1)$ and hence between the algebras $P(1),P(j+1)$.

\begin{definition}
Let $\prec$ be any monomial ordering of the subalgebra $P(1)\subset P$
and extend it to all subalgebras $P(j+1)$ $(j\geq 0)$ by the isomorphisms
$\rho(j)$. In other words, we put $j\cdot m\prec j\cdot n$ if and only if
$m\prec n$, for any $m,n\in M(1)$. Then, for all $m,n\in M,
m = m_{j_1}\cdots m_{j_k}, n = n_{j_1}\cdots n_{j_k}$ with $j_1 > \ldots > j_k$,
we define $m\prec' n$ if and only if $m_{j_s} = n_{j_s}$ and
$m_{j_t}\prec n_{j_t}$, for some $1\leq t\leq k$ and for all $1\leq s < t$.
By Proposition 3.7 in \cite{LS} one has that $\prec'$ is a monomial
$\N$-ordering that we call {\em place $\N$-ordering of $P$ induced by
a monomial ordering of $P(1)$}.
\end{definition}

Note that if $X$ is finite then $P(1)$ is a polynomial algebra in a finite
number of variables whose monomial orderings were classified in \cite{Ro}.
If $X$ is infinite, by Proposition \ref{higman} we have that the algebra $P(1)$
can be endowed anyway with monomial orderings provided that
$x_1(1)\prec x_2(1)\prec \ldots$.

An important feature of the place $\N$-orderings is that they are
compatible with some special grading of $P$ which is in turn compatible
with the action of $\N$. Denote $\hN = \{-\infty\}\cup\N$.

\begin{definition}
Let $\w:M\to\hN$ be the unique mapping such that
\begin{itemize}
\item[(i)] $\w(1) = -\infty$;
\item[(ii)] $\w(m n) = \max(\w(m),\w(n))$, for any $m,n\in M$;
\item[(iii)] $\w(x_i(j)) = j$, for all $i,j\geq 1$.
\end{itemize}
We call $\w$ the {\em weight function} of $P$. If $P_{(i)}\subset P$ is
the subspace spanned by all monomials of weight $i$ then
$P = \bigoplus_{i\in\hN} P_{(i)}$ is grading of $P$ over the idempotent
commutative monoid $(\hN,\max)$. Clearly, one has that $i\cdot P_{(j)}
\subset P_{(i+j)}$, for all $i,j$.
\end{definition}

\begin{definition}
Let $\prec$ be a monomial $\N$-ordering of $P$. We say that $\prec$ is a
{\em weighted ordering} if $\w(m) < \w(n)$ implies that $m\prec n$, for all
$m,n\in M$.
\end{definition}

By Proposition 5.11 in \cite{LS} one has that all place $\N$-orderings
are weighted ones. Note also that for multilinear monomials $m\in M\cap L$
one has that $\w(m) = \deg(m)$.

\begin{definition}
Let $\prec$ be a well-ordering of $W = \Mon(F)$. We call $\prec$ a {\em monomial
ordering} of $F$ if $m\prec n$ implies that $u m v\prec u n v$, for all
$m,n,u,v\in W$. In particular, we say that $\prec$ is a {\em graded ordering}
if $\deg(m) < \deg(n)$ implies that $m\prec n$, for any $m,n\in W$.
\end{definition}

\begin{theorem}
\label{freemord}
Let $\prec$ be a weighted $\N$-ordering of $P$ and define a total ordering
$\prec'$ of $W$ by putting $m\prec' n$ if and only if $\iota(m)\prec\iota(n)$,
for all $m,n\in W$. Then, the ordering $\prec'$ is a graded monomial
ordering of $F$ that we call {\em induced by $\prec$}.
\end{theorem}

\begin{proof}
It is clear that $\prec'$ is a well-ordering since the same holds
for the restriction of $\prec$ to $M\cap L$. Let $m',n',u',v'\in W$
and denote by $m,n,u,v\in M\cap L$ their images under the map $\iota$.
If $\deg(m') < \deg(n')$ then $\w(m) < \w(n)$ and hence $m\prec n$,
that is, one has that $m'\prec' n'$. Assume now $m'\prec' n'$.
If $\deg(m') < \deg(n')$ we have that $\deg(u' m' v') < \deg(u' n' v')$
and hence $u' m' v'\prec' u' n' v'$. If $d' = \deg(m') = \deg(n')$ and
$d = \deg(u')$ one obtains that $d\cdot m\prec d\cdot n$ since $\prec$
is an $\N$-ordering. We conclude that $\iota(u' m' v') =
u (d\cdot m)((d+d')\cdot v)\prec u (d\cdot n)((d+d')\cdot v) =
\iota(u' n' v')$, that is, we have that $u' m' v'\prec' u' n' v'$.
\end{proof}

The above result implies that a class of graded monomial orderings
of $F = K \langle X \rangle$ can be obtained from the class of weighted
$\N$-orderings of $P$ by restriction to $L$. In particular,
one has the following result.

\begin{corollary}
Let $\prec$ be any monomial ordering of $P(1)$ and extend it to a place
$\N$-ordering of $P$. Moreover, denote by $\prec'$ the graded monomial
ordering of $F$ induced by $\prec$ according to Theorem \ref{freemord}.
Then $\prec'$ is the graded right lexicographic order, that is,
for any $m = x_{i_1}\cdots x_{i_k}, n = x_{j_1}\cdots x_{j_k}\in W$
one has $m\prec' n$ if and only if $k < l$ or $k = l, i_s = j_s$ and
$i_t < j_t$, for some $1\leq t\leq k$ and for all $t < s\leq k$.
\end{corollary}

\begin{proof}
Note that if $X$ is an infinite set then necessarily
$x_1(i)\prec x_2(i)\prec \ldots$ and $x_1\prec' x_2\prec' \ldots$
because $\prec,\prec'$ are well-orderings. Then, one has that
$\iota(m) = x_{i_k}(k)\cdots x_{i_1}(1), \iota(n) =
x_{j_k}(k)\cdots x_{j_1}(1)$ and $\iota(m)\prec\iota(n)$ if and only if
$x_{i_s}(s) = x_{j_s}(s)$ and $x_{i_t}(t)\prec x_{j_t}(t)$, that is,
$i_s = j_s$ and $i_t < j_t$, for some $1\leq t\leq k$
and for all $t < s\leq k$.
\end{proof}

We start now introducing \Gr\ bases in the context of $\N$-ideals.
Fix $\prec$ any $\N$-ordering of $P$. Let $f = \sum_i c_i m_i\in P$ with
$m_i\in M, c_i\in K,c_i\neq 0$. We denote $\lm(f) = m_k =
\max_\prec\{m_i\}$, $\lc(f) = c_k$ and $\lt(f) = \lc(f)\lm(f)$.
Let $f,g\in P,f,g\neq 0$ and put $\lt(f) = c m, \lt(g) = d n$
with $m,n\in M$ and $c,d\in K$. If $l = \lcm(m,n)$ we define
as usual the {\em S-polynomial}
\[
\spoly(f,g) = (l/c m) f - (l/d n) g.
\]
Finally, if $G\subset P$ then we put $\lm(G) = \{\lm(f) \mid f\in G,f\neq 0\}$
and we denote by $\LM(G)$ the ideal of $P$ generated by $\lm(G)$.
The following results were proved in \cite{LSL1,LSL2}.

\begin{proposition}
Let $G\subset P$. Then $\lm(\N\cdot G) = \N\cdot \lm(G)$.
In particular, if $I$ is an $\N$-ideal of $P$ then $\LM(I)$
is also $\N$-ideal.
\end{proposition}

\begin{definition}
Let $I\subset P$ be an $\N$-ideal and $G\subset I$. We call $G$
a {\em \Gr\ $\N$-basis} of $I$ if $\lm(G)$ is an $\N$-basis of $\LM(I)$.
In other words, $\N\cdot G$ is a \Gr\ basis of $I$ as an ideal of $P$.
\end{definition}

\begin{definition}
Let $f\in P,f\neq 0$ and $G\subset P$. If $f = \sum_i f_i g_i$ with
$f_i\in P,g_i\in G$ and $\lm(f)\succeq\lm(f_i)\lm(g_i)$ for all $i$,
we say that {\em $f$ has a \Gr\ representation with respect to $G$}.
\end{definition}

\begin{theorem}
\label{Ncrit}
Let $G$ be an $\N$-basis of an $\N$-ideal $I\subset P$.
Then, $G$ is a \Gr\ $\N$-basis of $I$ if and only if for all
$f,g\in G,f,g\neq 0$ and for any $i\geq 0$, the S-polynomial
$\spoly(f, i\cdot g)$ has a \Gr\ representation with respect
to $\N\cdot G$.
\end{theorem}

For the sake of completeness, we recall also the notion of \Gr\ bases
for ideals of the free associative algebra. For any subset $G\subset F$,
define $\lm(G)$ and $\LM(G)$ as we have done for $P$.

\begin{definition}
Let $I\subset F$ be an ideal and $G\subset I$. We call $G$ a {\em \Gr\ basis}
of $I$ if $\lm(G)$ is a basis of $\LM(I)$. In other words, for any $f\in I,
f\neq 0$ one has that $\lm(f) = u \lm(g) v$, for some $g\in G,g\neq 0$ and
$u,v\in W$.
\end{definition}

From now on, {\em assume} that $P$ is endowed with a weighted $\N$-ordering
and $F$ is endowed with the induced graded monomial ordering. By abuse
of notation, we will denote both these orderings as $\prec$. We mention finally
the following key result proved in \cite{LSL1} for \Gr\ $\N$-bases
of letterplace ideals.

\begin{theorem}
\label{GBlettcorr}
Let $I\subset F$ be a graded ideal and denote by $J\subset P$ its letterplace
analogue. If $G$ is a multihomogeneous \Gr\ $\N$-basis of $J$ then
$\iota^{-1}(G\cap L)$ is a homogeneous \Gr\ basis of $I$.
\end{theorem}

This result together with the Theorem \ref{Ncrit} implies the following
algorithm for the computation of homogeneous noncommutative \Gr\ bases which
is alternative to the classical method developed in \cite{Gr, Mo, Uf1, Uf2}.

\suppressfloats[b]
\begin{algorithm}\caption{\HFreeGBasis}
\begin{algorithmic}[0]
\State \text{Input:} $H$, a homogeneous basis of a graded ideal $I\subset F$.
\State \text{Output:} $\iota^{-1}(G)$, a homogeneous \Gr\ basis of $I$.
\State $G:= \iota(H)$;
\State $B:= \{(f,g) \mid f,g\in G\}$;
\While{$B\neq\emptyset$}
\State choose $(f,g)\in B$;
\State $B:= B\setminus \{(f,g)\}$;
\ForAll{$i\geq 0$ \st\ $\gcd(\lm(f),\lm(i\cdot g))\neq 1,
\lcm(\lm(f),\lm(i\cdot g))\in L$}
\State $h:= \Reduce(\spoly(f,i\cdot g), \N\cdot G)$;
\If{$h\neq 0$}
\State $B:= B\cup\{(h,h),(h,k),(k,h),\mid k\in G\}$;
\State $G:= G\cup\{h\}$;
\EndIf;
\EndFor;
\EndWhile;
\State \Return $\iota^{-1}(G)$.
\end{algorithmic}
\end{algorithm}

The function $\Reduce$ is given by the following standard routine.

\suppressfloats[b]
\begin{algorithm}\caption{\Reduce}
\begin{algorithmic}[0]
\State \text{Input:} $G\subset P$ and $f\in P$.
\State \text{Output:} $h\in P$ such that $f - h\in\langle G\rangle$
and $h = 0$ or $\lm(h)\notin\LM(G)$.
\State $h:= f$;
\While{ $h\neq 0$ and $\lm(h)\in\LM(G)$ }
\State choose $g\in G,g\neq 0$ such that $\lm(g)$ divides $\lm(h)$;
\State $h:= h - (\lt(h)/\lt(g)) g$;
\EndWhile;
\State \Return $h$.
\end{algorithmic}
\end{algorithm}

Note that the iteration ``for all $i\geq 0$ \st\ $\ldots$''
in the procedure $\HFreeGBasis$ runs over a finite
number of integers since condition $\gcd(\lm(f),\lm(i\cdot g))\neq 1$
implies that $i < \w(f) = \deg(f)$. Moreover, by multihomogeneity of the 
elements of $P$ involved in the computation, one has that the condition
$\lcm(\lm(f),\lm(i\cdot g))\in L$ is equivalent to require that the element
$h = \Reduce(\spoly(f,i\cdot g), \N\cdot G)$ is multilinear.
Note finally that there are clearly a finite number of elements
of the infinite set $\lm(\N\cdot G) = \N\cdot\lm(G)$ that may participate
to such reduction.
Even if one assumes that $X$ is a finite set, owing to non-Noetherianity
of the free associative algebra $F = K\langle X \rangle$ or of the polynomial
algebra $P = K[X(\N^*)]$ that has an infinite number of variables,
we may have that the ideal $I\subset F$ is finitely generated but the leading
monomial ideal $\LM(I)$ is not such, that is, the \Gr\ bases of $I$ are
infinite sets. In other words, we do not have general termination 
for the algorithm $\HFreeGBasis$ but termination is clearly provided
for truncated computations up to some fixed degree, assuming
that the ideal $I$ is finitely generated within such degree.
For more details about the above algorithm we refer to \cite{LSL1,LSL2}.


\section{\Gr\ $L$-bases and $L$-saturation}

The fact that the letterplace ideals are $\N$-generated by multilinear
elements and Theorem \ref{GBlettcorr} suggest that for such ideals
one needs a notion of \Gr\ basis that involves only multilinear
elements.

\begin{definition}
Let $J$ be an $L$-ideal of $P$ and let $H\subset J\cap L$ be a subset
of multilinear elements. If $H$ is an $\N$-basis of $J$ then
we call $H$ a {\em $L$-basis} or {\em multilinear $\N$-basis} of $J$.
\end{definition}

\begin{definition}
\label{GLB}
Let $J\subset P$ be an $L$-ideal and denote $\LM_L(J) =
\langle \lm(f)\mid f\in J\cap L \rangle_\N$. Let $G\subset J\cap L$
be a subset of multilinear elements. We call $G$ a {\em \Gr\ $L$-basis}
or {\em \Gr\ multilinear $\N$-basis} of $J$ if $\lm(G)$ is an $\N$-basis
of $\LM_L(J)$, that is, for all multilinear elements $f\in J\cap L$ one has
that $i\cdot \lm(g)$ divides $\lm(f)$, for some $g\in G$ and $i\geq 0$.
Clearly, all \Gr\ $L$-bases are also $L$-bases of letterplace ideals.
\end{definition}

If $I$ is a graded ideal of $F$ and $J\subset P$ is its letterplace analogue,
by Theorem \ref{GBlettcorr} one has that $G\subset J\cap L$ is a
\Gr\ $L$-basis of $J$ if and only if $\iota^{-1}(G)$ is a homogeneous
\Gr\ basis of $I$. In this sense, we may say that \Gr\ $L$-bases are
``letterplace analogues'' of homogeneous \Gr\ bases of the free associative
algebra. Another interesting feature of \Gr\ $L$-bases is that they can
be obtained as complete multihomogeneous \Gr\ $\N$-bases of suitable ideals.

\begin{definition}
Denote $N = \langle x_i(1)x_j(1)\mid i,j\geq 1 \rangle_\N\subset P$.
A monomial $m = x_{i_1}(j_1)\cdots x_{i_d}(j_d)\in M$ is said to be
{\em normal modulo $N$} if $j_1\neq \ldots \neq j_d$. A polynomial
$f\in P$ is {\em in normal form modulo $N$} if all its monomials
are normal modulo $N$.
\end{definition}

\begin{definition}
Let $J\subset P$ be an $\N$-ideal containing $N$ and let $G\subset J$ be
a subset of polynomials that are in normal form modulo $N$. We say that
$G$ is a {\em \Gr\ $\N$-basis of $J$ modulo $N$} if
$G\cup\{x_i(1)x_j(1)\mid i,j\geq 1\}$ is a \Gr\ $\N$-basis of $J$.
\end{definition}

\begin{theorem}
\label{ntrick}
Let $J$ be an $L$-ideal of $P$ and let $G\subset J\cap L$. Then $G$ is
a \Gr\ $L$-basis of $J$ if and only if $G$ is a multihomogeneous \Gr\
$\N$-basis of $J + N$ modulo $N$.
\end{theorem}

\begin{proof}
It is sufficient to prove that there is a \Gr\ $\N$-basis of $J + N$
modulo $N$ whose elements are all multilinear. Then, consider to apply
the Buchberger algorithm to an $L$-basis of $J$. By the product
criterion and multihomogeneity of the computation, it is clear that
for the monomials $m = \lcm(\lm(f),i\cdot\lm(g))$ where $f,g$ are elements
of the current $\N$-basis, one has that either $m$ is multilinear or $m\in N$.
\end{proof}

Note that the above result, together with the comments after Definition
\ref{GLB}, provide another insight into the subtle relationships between
noncommutative structures and their commutative analogues subjected
to the action of the monoid $\N$.

Let us extend now the results of Section 4 and the previous ones to the
algebras $\bF,\bP$. In what follows, {\em assume} the polynomial algebra
$\bP$ be endowed with a place $\N$-ordering which is induced by a monomial ordering
of $\bP(1)$ such that $t(1)\prec x_1(1)\prec x_2(1)\prec\ldots$. Therefore,
the free associative algebra $\bF$ is provided with the graded right
lexicographic ordering such that $t\prec x_1\prec x_2\prec\ldots$.
One obtains immediately the following result.

\begin{proposition}
The elements $\biota([t,x_i]) = t(1) x_i(2) - x_i(1) t(2)$ $(i\geq 1)$
are a \Gr\ $L$-basis of the $L$-ideal $D$, that is, the
commutators $[t,x_i]$ are a homogeneous \Gr\ basis of the graded
ideal $C$. Then, a multilinear element $f\in\bL$ is said to be {\em in normal
form modulo $D$} if it is such with respect to the above \Gr\ $L$-basis.
\end{proposition}

\begin{definition}
Let $J\subset\bP$ be an $L$-ideal which contains $D$ and let $G\subset J\cap\bL$
be a subset of multilinear elements that are in normal form modulo $D$.
We say that $G$ is a {\em \Gr\ $L$-basis of $J$ modulo $D$} if
$G\cup\{\biota([t,x_i])\mid i\geq 1\}$ is a \Gr\ $L$-basis of $J$.
\end{definition}

A natural characterization of the $L$-saturation of a letterplace ideal
containing $D$ is the following one.

\begin{lemma}
\label{lsatchar}
Let $J\subset\bP$ be an $L$-ideal containing $D$. Then, a \Gr\ $L$-basis
of $\Sat_L(J)$ modulo $D$ is given by the elements $\psi(f)^*$,
for all $f\in J\cap\bL$ that are in normal form modulo $D$.
\end{lemma}

\begin{proof}
It is sufficient to note that if $f\in\bL$ is in normal form
modulo $D$ then $\psi(f)\in V$ and $g = \psi(f)^*\in\bL$. Moreover,
it is clear that $f = \prod_{d<k\leq d'} t(k) g$ where
$\deg(f) = d'\geq d = \deg(g)$.
\end{proof}

\begin{theorem}
\label{satgb}
Let $J\subset\bP$ be an $L$-ideal which contains $D$ and denote by $J' = \Sat_L(J)$
its $L$-saturation. Moreover, let $G$ be a \Gr\ $L$-basis of $J$ modulo $D$.
Then $G' = \psi(G)^* = \{ \psi(g)^* \mid g\in G \}$ is a \Gr\ $L$-basis
of $J'$ modulo $D$.
\end{theorem}

\begin{proof}
Note that if $f'\in\bL$ is in normal form modulo $D$ and $f = \psi(f')\in V$
then $\lm(f^*) = \lm(f)\in M$ by definition of the monomial ordering of $\bP$.
Now, let $f'\in J\cap\bL$ be an element in normal form modulo $D$.
Hence, there is $g'\in G$ and $h\geq 0$ such that $h\cdot\lm(g')$ divides
$\lm(f')$. Put $f = \psi(f'), g = \psi(g')$ and $m_i = \prod_{0<j\leq i} t(j)$. 
Then, one has that $f' = f^*(i\cdot m_j), g' = g^*(k\cdot m_l)$ where
$i = \deg(f), k = \deg(g)$ and $j,l\geq 0$. From $h\cdot\lm(g')$ divides
$\lm(f')$ if follows that $h + k\leq i$ and hence $k\cdot \lm(g^*)$ divides
$\lm(f^*)$. We conclude that $\psi(G)^*$ is a \Gr\ $L$-basis of $\Sat_L(J)$
modulo $D$.
\end{proof}

From the above result one obtains immediately an algorithm for computing
\Gr\ $L$-bases of $L$-saturated letterplace ideals of $\bP$ containing $D$.
This is especially relevant since such bases are in correspondence with
homogeneous \Gr\ bases of saturated ideals of $\bF$ containing $C$.
In fact, the \Gr\ bases of any ideal $I\subset F$ are in correspondence
with the homogeneous ones of its homogenization $I^*$. 

\begin{definition}
A homogeneous element $f\in\bF$ is said to be {\em in normal form modulo $C$}
if it is such with respect to the \Gr\ basis $\{[t,x_i]\mid i\geq 1\}$.
In other words, the multilinear element $\biota(f)\in\bL$ is in normal form
modulo $D$.
\end{definition}

Note that $\biota(f^*) = \iota(f)^*$, for all $f\in F, f\neq 0$. Moreover,
if $f\in\bF$ is a homogeneous element in normal form modulo $C$ then
we have also that $\iota(\varphi(f)) = \psi(\biota(f))$.

\begin{definition}
Let $I\subset\bF$ be a graded ideal which contains $C$ and let $G\subset I$
be a subset of homogeneous elements which are in normal form modulo $C$.
We say that $G$ is a {\em \Gr\ basis of $I$ modulo $C$} if
$G\cup\{[t,x_i]\mid i\geq 1\}$ is a \Gr\ basis of $I$.
In other words, $\biota(G)\subset\bL$ is a \Gr\ $L$-basis modulo $D$
of the letterplace analogue of $I$.
\end{definition}

The following result can be found also in \cite{LiSu,Uf3}.

\begin{theorem}
\label{homgb}
Let $I\subset F$ be any ideal and let $G$ be any \Gr\ basis of $I$.
Then $G^* = \{ g^* \mid g\in G \}$ is a homogeneous \Gr\ basis of $I^*$
modulo $C$. Moreover, one has that $\lm(G^*) = \lm(G)$.
\end{theorem}

\begin{proof}
Let $f'\in I^*$ be a homogeneous element in normal form modulo $C$
and put $f = \varphi(f')$. Then $f' = f^* t^i$ for some $i\geq 0$ and
$\lm(f) = u \lm(g) v$ for some $g\in G$ and $u,v\in W$. Since
$\lm(f^*) = \lm(f), \lm(g^*) = \lm(g)$ we conclude that $\lm(f') =
u \lm(g^*) v t^i$.
\end{proof}

The above result and Theorem \ref{satgb} imply an alternative algorithm
to compute \Gr\ bases of nongraded noncommutative ideals of the free
associative algebra via homogeneous commutative computations
in their extended letterplace analogues.

\suppressfloats[b]
\begin{algorithm}\caption{FreeGBasis}
\begin{algorithmic}[0]
\State \text{Input:} $H$, a basis of an ideal $I\subset F$.
\State \text{Output:} $\varphi(\biota^{-1}(G))$, a \Gr\ basis of $I$.
\State $G:= \biota(H^* \cup \{[t,x_i] \mid i\geq 1\})$;
\State $B:= \{(f,g) \mid f,g\in G\}$;
\While{$B\neq\emptyset$}
\State choose $(f,g)\in B$;
\State $B:= B\setminus \{(f,g)\}$;
\ForAll{$i\geq 0$ \st\ $\gcd(\lm(f),\lm(i\cdot g))\neq 1,
\lcm(\lm(f),\lm(i\cdot g))\in\bL$}
\State $h:= \Reduce(\spoly(f,i\cdot g), \N\cdot G)$;
\If{$h\neq 0$}
\State $h:= \psi(h)^*$
\State $B:= B\cup\{(h,h),(h,k),(k,h)\mid k\in G\}$;
\State $G:= G\cup\{h\}$;
\EndIf;
\EndFor;
\EndWhile;
\State \Return $\varphi(\biota^{-1}(G))$.
\end{algorithmic}
\end{algorithm}

\newpage
\begin{theorem}
The algorithm \FreeGBasis\ is correct.
\end{theorem}

\begin{proof}
Let $J\subset\bP$ be the extended letterplace analogue of $I\subset F$.
At each step of the procedure \FreeGBasis, the set $G$ is clearly an $L$-basis
of an ideal $J'\subset \bP$ containing $D$ such that $\Sat_L(J') = J$. Moreover,
since the elements $\biota([t,x_i])$ initially belong to $G$ we have the automatic
normalization modulo $D$ of the elements obtained during the computation.
Recall now that if $h\in\bL$ is a multilinear element which is in normal form
modulo $D$ then $h' = \psi(h)^*$ divides $h$. This implies that if an S-polynomial
can be reduced to zero by adding $h$ to the basis $G$, the same holds if
we substitute $h$ with $h'$. In case of termination, one has therefore that
the set $G$ is a \Gr\ $L$-basis of $J'$ whose elements satisfy $h = \psi(h)^*$.
By Theorem \ref{satgb} we conclude that $J'$ is $L$-saturated, that is,
one has that $J' = J$. Then $G' = \biota^{-1}(G)$ is homogeneous \Gr\ basis
of $I^*$, that is, $\varphi(G')$ is a \Gr\ basis of $I$ by Theorem \ref{homgb}.
\end{proof}

Note that the above algorithm has neither general termination nor
just termination up to some fixed degree $d$. The reason is that even
if all computations are homogeneous, because of the saturation
$h = \psi(h)^*$ that may decrease the current degree we cannot be sure at some
suitable step that we will not get additional elements of degree $\leq d$
in the steps that will follow. This agrees with the well known fact that
the word-problem is generally undecidable for nongraded associative algebras
even if these are finitely generated. Nevertheless, if an ideal of the
free associative algebra has a finite \Gr\ basis then the algorithm
\FreeGBasis\ is able to compute it in a finite number of steps.

\begin{definition}
Let $G\subset F$ be any subset. We call $G$ a {\em minimal \Gr\ basis}
if $\lm(G)$ is a minimal basis of $\LM(G)$, that is, $\lm(f)\neq u \lm(g) v$,
for all $f,g\in G,f\neq g$ and for any $u,v\in W$.
\end{definition}

By the choice of the monomial ordering of $\bF$ and the property that
the elements are kept in normal form modulo $C$ we have clearly that
if $G'$ is a minimal \Gr\ basis of $I^*$ modulo $C$ then $\varphi(G')$
is also a mimimal \Gr\ basis of $I$ since $\lm(G') = \lm(\varphi(G'))$.
This is the main advantage to compute on the fly the homogenization
$I^*$ instead of working with any graded ideal $C\subset I'\subset\bF$
such that $\varphi(I') = I$. In fact, the ideal $I'$ may have an infinite
minimal \Gr\ basis even if $I$ has a finite one and more generally this
basis has elements in higher degrees than the basis of $I^*$. In other
words, to compute without saturation is usually very inefficient.
Such strategy is described in \cite{Uf3} in the context of classical
algorithm and called ``rabbit strategy'' or ``cancellation rule''.

Note that actual computations with the algorithm \FreeGBasis\ are
performed by bounding the weight of the variables of $P$, that is,
in a (Noetherian) polynomial algebra with a finite number
of variables. This may result in an incomplete computation because
some of the S-polynomials may be not defined owing to this bound.
Since the S-polynomials $s = \spoly(f,i\cdot g)$ such that
$\gcd(\lm(f),\lm(i\cdot g))\neq 1$ that are considered in the
procedure are multilinear elements, it is clear that $\w(s) =
\deg(s)\leq 2d - 1$ where $d = \max\{\deg(f)\mid f\in G\}$ and $G$
is the current basis. We conclude that an actual computation
is certified complete if the weight bound for the variables
of $P$ is $\geq 2d - 1$, where $d$ is the maximal degree occuring
in the output generators.


\section{An illustrative example}

With the aim of showing a concrete computation with the algorithm
\FreeGBasis, we present here a simple application to finitely presented
(noncommutative) groups. Consider the symmetric group $S_3$ that can be
presented, as a Coxeter group, in the following way
\[
S_3 = \langle x,y\mid x^2 = y^2 = (x y)^3 = 1 \rangle.
\]
Define the free associative algebra $F = K\langle x, y \rangle$ and
consider the elements
\[
f_1 = x^2 - 1, f_2 = y^2 - 1, f_3 = (x y)^3 - 1\in F.
\]
Then, the group algebra $K S_3$ is clearly isomorphic to the quotient
algebra $F/I$ where $I = \langle f_1, f_2, f_3 \rangle$. A next step
is to consider the free commutative $\N$-algebra
$\bP = K[x(1),y(1),t(1),x(2),y(2),t(2),\ldots]$ and to encode
the noncommutative algebra $F/I$ in the letterplace way, that is,
by defining the extended letterplace analogue $J\subset \bP$
of the two-sided ideal $I\subset F$. As explained in the comments
at the end of Section 3, we consider therefore the polynomials
\[
\begin{array}{c}
d_1 = \biota([t,x]) = t(1)x(2) - x(1)t(2),
d_2 = \biota([t,y]) = t(1)y(2) - y(1)t(2), \\
g_1 = \biota(f_1^*) = x(1)x(2) - t(1)t(2),
g_2 = \biota(f_2^*) = y(1)y(2) - t(1)t(2), \\
g_3 = \biota(f_3^*) = x(1)y(2)x(3)y(4)x(5)y(6) - t(1)t(2)t(3)t(4)t(5)t(6)\in\bP \\
\end{array}
\]
and we define the $L$-ideal $J' = \langle d_1,d_2,g_1,g_2,g_3 \rangle_\N$.
In fact, one has that $J = \Sat_L(J')$ and to perform this ideal operation
one needs a \Gr\ basis computation. Then, we fix the lexicographic monomial
ordering on $\bP$ with
\[
\ldots\succ x(2)\succ y(2)\succ t(2)\succ x(1)\succ y(1)\succ t(1)
\]
which is clearly a place $\N$-ordering inducing the graded right
lexicographic ordering on $F$ with $x\succ y$. Then, to compute $\Sat_L(J')$
one has to reduce the multilinear S-polynomials between generators and
performing the saturation of new generators arising by such reductions.
At the end of computation, whenever $I$ admits a finite \Gr\ basis, one
obtains a (saturated) \Gr\ $L$-basis $G\subset J$, that is, a \Gr\ basis
$\varphi(\biota^{-1}(G))$ of $I$, as prescribed by the algorithm
\FreeGBasis.

First of all, note that no multilinear S-polynomials are defined for
the elements $d_i$. Moreover, it is easy to see that all multilinear
S-polynomials between $d_i$ and any saturated element can be reduced
to zero. For instance, one has that the S-polynomial
\[
\spoly(d_1,1\cdot g_1) = - x(1)t(2)x(3) + t(1)t(2)t(3)
\]
is reduced modulo $1\cdot d_1$ to the element $x(1)x(2)t(3) - t(1)t(2)t(3) =
g_1 t(3)$ which clearly goes to zero modulo $g_1$.

Consider now the S-polynomial $\spoly(g_1,1\cdot g_1) = 
- t(1)t(2)x(3) + x(1)t(2)t(3)$ that can be clearly reduced to zero modulo
the set $\N\cdot d_1$. In a similar way, one obtains that
$\spoly(g_2,1\cdot g_2)$ reduces to zero. Then, we define the S-polynomial
\[
\spoly(g_3,5\cdot g_2) = - t(1)t(2)t(3)t(4)t(5)t(6)y(7) +
x(1)y(2)x(3)y(4)x(5)t(6)t(7)
\]
which is reduced modulo the set $\N\cdot d_2$ to the element
\[
g'_4 = x(1)y(2)x(3)y(4)x(5)t(6)t(7) - y(1)t(2)t(3)t(4)t(5)t(6)t(7).
\]
This polynomial cannot be further reduced by the current $\N$-basis
and hence one adds to this set the corresponding saturated element
\[
g_4 = \psi(g'_4)^* = x(1)y(2)x(3)y(4)x(5) - y(1)t(2)t(3)t(4)t(5).
\]
Then, we consider $\spoly(g_3,g_4) = y(1)t(2)t(3)t(4)t(5)y(6) -
t(1)t(2)t(3)t(4)t(5)t(6)$ that can be reduced to zero modulo
$\N\cdot\{d_2,g_2\}$. Consider now the next S-polynomial
\[
\spoly(g_1,1\cdot g_4) = - t(1)t(2)y(3)x(4)y(5)x(6) + x(1)y(2)t(3)t(4)t(5)t(6)
\]
By applying the set $\N\cdot\{d_1,d_2\}$ one obtains the element
\[
g'_5 = y(1)x(2)y(3)x(4)t(5)t(6) - x(1)y(2)t(3)t(4)t(5)t(6)
\]
and hence its saturation
\[
g_5 = \psi(g'_5)^* = y(1)x(2)y(3)x(4) - x(1)y(2)t(3)t(4)
\]
enters the current $\N$-basis. Then, one considers
\[
\spoly(g_2,1\cdot g_5) = - t(1)t(2)x(3)y(4)x(5) + y(1)x(2)y(3)t(4)t(5)
\]
which is reduced modulo $\N\cdot\{d_1,d_2\}$ to the element
\[
g'_6 = x(1)y(2)x(3)t(4)t(5) - y(1)x(2)y(3)t(4)t(5)
\]
and therefore the corresponding saturated element
\[
g_6 = \psi(g'_6)^* = x(1)y(2)x(3) - y(1)x(2)y(3)
\]
is appended to the $\N$-basis of the current $L$-ideal.
All remaining S-polynomials reduce to zero which means that such ideal
is $L$-saturated and therefore coincides with $J = \Sat_L(J')$. Note that
the sequence of leading monomials of the polynomials $g_i$ is
\[
\begin{array}{c}
\lm(g_1) = x(1)x(2), \lm(g_2) = y(1)y(2), \lm(g_3) = x(1)y(2)x(3)y(4)x(5)y(6), \\
\lm(g_4) = x(1)y(2)x(3)y(4)x(5), \lm(g_5) = y(1)x(2)y(3)x(4),
\lm(g_6) = x(1)y(2)x(3)
\end{array}
\]
and one has that $\lm(g_6)$ divides $\lm(g_3),\lm(g_4)$ and $1\cdot \lm(g_6)$
divides $\lm(g_5)$. We conclude that a minimal \Gr\ $L$-basis of the ideal $J$
is given by the set $G = \{d_1,d_2,g_1,g_2,g_6\}$. Because $J\subset\bP$
is exactly the extended letterplace analogue of the two-sided ideal
$I\subset F$, we obtain that the set $\{x^2 - 1, y^2 - 1, x y x - y x y\}$
is a minimal \Gr\ basis of $I$ with respect to graded right lexicographic
ordering. In other words, we have found the canonical presentation
\[
S_3 = \langle x,y \mid x^2 = y^2 = 1, x y x = y x y \rangle
\]
showing that this group is a quotient of the braid group
$B_3 = \langle x,y \mid x y x = y x y \rangle$.


\section{Implementations and testing}

In this section we present an experimental implementation of the
algorithm \FreeGBasis\ that has been developed in the language of Maple.
We have obtained such implementation by modifying the algorithm
\SigmaGBasis\ introduced and experimented in \cite{LS} for the
computation of \Gr\ bases for finite difference ideals. Precisely,
the letterplace computations are a special case of the ordinary
difference ones. The main modifications to obtain \FreeGBasis\
consist in adding the commutators $[t,x_i]$ to the elements introduced
by homogenizing the initial noncommutative generators and in encoding
all such elements in the letterplace way. Moreover, it is necessary
to add to the procedure the ``multilinearity criterion'', that is,
the condition $\lcm(\lm(f),\lm(i\cdot g))\in\bL$ when considering
the S-polynomial $\spoly(f,i\cdot g)$. Finally, one has to implement
the saturation of the elements that are obtained by reducing these
S-polynomials. Note that according to Theorem \ref{ntrick},
the multilinearity criterion, which is essential to have tractable computations,
can be obtained simply by adding the set of monomials $\mathcal{N} =
\{x_i(1)x_j(1), t(1)^2, t(1)x_i(1)\}$ to the initial letterplace basis.
This option is a useful trick if one wants to obtain the algorithm
\FreeGBasis\ by means of a standard implementation of the Buchberger
procedure for commutative \Gr\ bases.

To the purpose of studying the impact of different strategies used
in \FreeGBasis, we have tested also two variants of this algorithm
that are indicated in the examples with the suffix {\em noc}
(no-criterion) and {\em bas} (basic). Both these variants make use
of the saturation step $h := \psi(h)^*$ since it is well known that
mere homogenization of the initial generators is generally inefficient
and may lead to an infinite \Gr\ basis for the corresponding graded
but not saturated ideal even if the input ideal have a finite one
\cite{Uf3}. The variant {\em noc} is obtained simply by suppressing
the ``shift criterion'', that is, all S-polynomials
$\spoly(i\cdot f,j\cdot g)$ ($i,j\in\N)$ are considered for reduction.
In the variant {\em bas} we suppress also the shifting of the new
generators obtained from the reduction of the S-polynomials. In other
words, one applies shift operators just to the input letterplace generators.
This is correct since the different shifted versions of the generators
that are necessary to the reduction process will be created in any case
from the S-polynomials provided that the shift criterion is off.
Up to the saturation step, the basic version can be obtained therefore
by applying the Buchberger algorithm to the set of shifted elements
of the initial letterplace basis joined to the set of monomials
$\mathcal{N}$. We apply this trick on some examples where no saturation
arises, in order to have computing times with standard routines of
\textsf{Singular} that estimate approximately the speed-up that
one may obtain moving from the Maple interpreter to the kernel of
a computer algebra system. Note that an implementation of noncommutative
\Gr\ bases in the library \textsf{LETTERPLACE} of \textsf{Singular}
is currently under development \cite{LSS}.

The monomial $\N$-ordering which is considered for the polynomial algebra
$\bP$ is the lexicographic ordering with 
\[
\ldots\succ
x_1(2)\succ\ldots\succ x_n(2)\succ t(2)\succ
x_1(1)\succ\ldots\succ x_n(1)\succ t(1)
\]
that is clearly a place $\N$-ordering. Then, one has that the free
associative algebra $\bF = F \langle x_1,\ldots,x_n,t \rangle$ is endowed
with the graded {\em left} lexicographic ordering with
$x_1\succ\ldots\succ x_n\succ t$ by means of a reversing letterplace
embedding $\biota':y_1\cdots y_d\mapsto y_d(1)\cdots y_1(d)$, where
$y_k = x_{i_k}$ or $y_k = t$.

The parameters that are considered in the experiments are the number of
respectively input generators, output \Gr\ generators, elements of
a minimal \Gr\ basis, pairs (S-polynomials) that are actually reduced
and saturations steps. The last parameter is the computing time which
is given in the format ``minutes:seconds''. Attached to the number of
elements of a basis, after the letter ``d'' we indicate the maximum degree
of such elements. Note that for the variant {\em bas} the input and
output numbers count all the shifted versions of the basis elements.
Moreover, for all variants we have that the pairs number includes
the initial generators since we actually treat them as S-polynomials
in order to interreduce. All examples have been computed with Maple 12
running on a server with a four core Intel Xeon at 3.16GHz and 64 GB RAM.

\newpage
\begin{center}
\begin{tabular}[t]{|l|c|c|c|c|c|c|}
\hline
Example         & gens   & gb     & min gb & pairs & sats & time  \\
\hline
g3332d10        &   8d8  &  52d8  & 29d5   & 665   & 33   & 01:15 \\
g3332d10-noc    &   8d8  &  52d8  & 29d5   & 1904  & 33   & 02:29 \\
g3332d10-bas    &  60d8  & 232d8  & 29d5   & 1001  & 142  & 00:35 \\
\hline
g444d10         &   7d6  &  95d7  & 51d5   & 2657  & 40   & 11:28 \\
g444d10-noc     &   7d6  &  95d7  & 51d5   & 4201  & 40   & 17:43 \\
g444d10-bas     &  47d6  & 578d9  & 51d5   & 2396  & 342  & 06:29 \\
\hline
heckeAd15      &  10d3  &  27d11 & 27d11  & 237   & 0    & 00:49 \\
heckeAd15-noc  &  10d3  &  27d11 & 27d11  & 1657  & 0    & 02:53 \\
heckeAd15-bas  & 136d3  & 950d15 & 27d11  & 3902  & 0    & 51:53 \\
\hline
heckeDd15      &  10d3  &  16d7  & 16d7   & 89    & 0    & 00:13 \\
heckeDd15-noc  &  10d3  &  16d7  & 16d7   & 783   & 0    & 00:47 \\
heckeDd15-bas  & 137d3  & 250d11 & 16d7   & 1028  & 0    & 00:57 \\
\hline
heckeEd10      &  21d3  &  50d10 & 50d10  & 396   & 0    & 01:00 \\
heckeEd10-noc  &  21d3  &  50d10 & 50d10  & 1528  & 0    & 02:28 \\
heckeEd10-bas  & 184d3  & 630d10 & 50d10  & 2730  & 0    & 11:29 \\
\hline
lie5d25         &   3d2  &  26d25 & 26d25  & 26    & 0    & 02:02 \\
lie5d25-noc     &   3d2  &  26d25 & 26d25  & 279   & 0    & 03:08 \\
lie5d25-bas     &  72d2  & 348d25 & 26d25  & 348   & 0    & 02:53 \\
\hline
lie7d5          &  10d2  &  40d3  & 21d2   & 181   & 11   & 00:41 \\
lie7d5-noc      &  10d2  &  40d3  & 21d2   & 368   & 11   & 00:55 \\
lie7d5-bas      &  40d2  & 908d5  & 21d2   & 1982  & 106  & $>$2h \\
\hline
templieb8d8     &  34d3  &  64d8  & 64d8   & 581   & 0    & 01:02 \\
templieb8d8-noc &  34d3  &  64d8  & 64d8   & 1721  & 0    & 02:18 \\
templieb8d8-bas & 226d3  & 336d8  & 64d8   & 1879  & 0    & 01:31 \\
\hline
templieb9d9     &  43d3  &  85d9  & 85d9   & 920   & 0    & 03:47 \\
templieb9d9-noc &  43d3  &  85d9  & 85d9   & 3189  & 0    & 09:41 \\
templieb9d9-bas & 330d3  & 512d9  & 85d9   & 3418  & 0    & 05:43 \\
\hline
\end{tabular}
\end{center}

The performance of the different variants of the algorithm \FreeGBasis\
have been studied on a test set based on presentations of relevant
classes of noncommutative algebras. The examples {\em g3332} and {\em g444}
refer to the presentation of group algebras of presented groups.
Precisely, such groups belong to the classes $G(l,m,n,q) =
\langle r,s\mid r^l, s^m, (rs)^n, [r,s]^q \rangle$ and $G(m,n,p) =
\langle a,b,c\mid a^m, b^n, c^p, (ab)^2, (bc)^2, (ca)^2,(abc)^2 \rangle$.
The examples {\em hecke} are the presentation of the Hecke algebras
defined by the following Coxeter matrices
\[
A =
\left(
\begin{array}{
@{\hskip 2pt}c@{\hskip 2pt}c@{\hskip 2pt}c@{\hskip 2pt}c@{\hskip 2pt}
}
1 & 3 & 2 & 3 \\
3 & 1 & 3 & 2 \\
2 & 3 & 1 & 3 \\
3 & 2 & 3 & 1 \\
\end{array}
\right);
D =
\left(
\begin{array}{
@{\hskip 2pt}c@{\hskip 2pt}c@{\hskip 2pt}c@{\hskip 2pt}c@{\hskip 2pt}
}
1 & 3 & 2 & 2 \\
3 & 1 & 3 & 3 \\
2 & 3 & 1 & 2 \\
2 & 3 & 2 & 1 \\
\end{array}
\right);
E =
\left(
\begin{array}{
@{\hskip 2pt}c@{\hskip 2pt}c@{\hskip 2pt}c@{\hskip 2pt}c@{\hskip 2pt}
c@{\hskip 2pt}c@{\hskip 2pt}
}
1 & 2 & 3 & 2 & 2 & 2 \\
2 & 1 & 2 & 3 & 2 & 2 \\
3 & 2 & 1 & 3 & 2 & 2 \\
2 & 3 & 3 & 1 & 3 & 2 \\
2 & 2 & 2 & 3 & 1 & 3 \\
2 & 2 & 2 & 2 & 3 & 1 \\
\end{array}
\right).
\]
For the noncommutative polynomials defining the relations of the considered
Hecke algebras, the quantity ``$q$'' is assumed a parameter. The examples
indicated as {\em lie} refer to the universal eveloping algebra
of two indecomposable nilpotent Lie algebras, namely
\[
\begin{array}{l}
\mbox{lie5}: [x_1,x_2] - x_3, [ x_1,x_3 ] - x_4, [ x_2,x_5 ] - x_4; \\
\\
\mbox{lie7}: [x_1,x_2] - x_3, [x_1,x_3] - x_4, [x_1,x_4] - x_5,
[x_1,x_5] - x_6, \\
\hspace{23.3pt}
[x_2,x_3] - \frac{1}{2} x_4 - \frac{1}{4} x_5 + \frac{1}{8} x_6 +
\frac{1}{2} x_7, [x_2,x_4] - \frac{1}{2} x_5 - \frac{1}{4} x_6, \\
\hspace{23.3pt}
[x_2,x_5] - x_6, [x_2,x_7] - \frac{1}{2} x_5 + \frac{1}{4} x_6,
[x_3,x_4] + \frac{1}{2} x_6, [x_3,x_7] - \frac{1}{2} x_6.
\end{array}
\]
Finally, the examples {\em templieb8, templieb9} are the defining relations
of the Temper\-ley-Lieb algebras \cite{KL} respectively in 7 and 8 variables.
The quantity ``$\delta$'' used in the definition of such algebras is considered
a parameter. In all the names of the tests, we indicate after the letter ``d''
the bounded degree within the computation is performed, that is, the maximal
weight which is allowed for the variables of $P$.

The experiments show in a sufficiently clear way that the standard
version of the algorithm \FreeGBasis\ is generally the most efficient one.
In fact, this procedure is able to decrease relevantly the number of S-polynomial
reductions that are usually time-consuming. For instance, this emerges
in a dramatic way for the example {\em lie7}. Note that for the examples
{\em g3332,g444} the basic variant results very competitive. This can
be explained as the result of a low cost for the S-polynomial reductions
(binomial generators) compared to the cost of applying shifting to
letterplace polynomials. As previously remarked, a noncommutative
\Gr\ basis computed up to a fixed weight is certified complete
if such bound is $\geq 2d-1$, where $d$ is the maximal degree of the
output generators. This happens for instance for the examples
{\em g3332,g444,heckeD} and {\em lie7}. In particular, one obtains
a computational proof that the ideal obtained by homogenizing
the relations defining the Hecke algebra of the example {\em heckeD}
is a saturated one.

The computing times obtained with the implementation of \FreeGBasis\
in the language of Maple are useful to evaluate the possible
different variants of this algorithm but they are not especially relevant 
when compared to other implementations of noncommutative \Gr\ bases
developed in the kernel of highly efficient computer algebra systems.
Among these fast implementations, one has to mention the one of
\textsf{Magma} \cite{BCP} that makes use of a noncommutative version
of the Faugere's F4 method. To the purpose of estimating the speed-up
that may be achieved with a kernel implementation, we have computed
the timings of some examples with the basic variant of \FreeGBasis\
obtained by using the function ``std'' of \textsf{Singular}
that implements the Buchberger algorithm. For the examples {\em heckeAd15,
heckeEd10} and {\em teli9d9} such computing times are respectively
0.26, 0.34 and 1.01 sec. Keeping into account that the variant {\em bas}
shows to be the less efficient, we believe that these data,
together with all experiments performed in \cite{LSL1,LSL2,LSS}, indicate
that letterplace approach is feasible for both the homogeneous and
inhomogeneous case.

\section{Conclusion and future directions}

The theory and the methods proposed in this paper and in the previous
ones \cite{LSL1,LSL2} proves that commutative and noncommutative \Gr\ bases
and the related algorithms can be unified in a general theory for
\Gr\ bases of commutative ideals that are invariant under the action
of suitables algebra endomorphisms \cite{BD,GLS,LSL2,LS}.

We believe that this idea will have not only consequences in the development
of new algorithmic methods but also in the reformulation in the letterplace
language of structures and problems of noncommutative nature. It is
sufficient in fact to mention that the notion of \Gr\ basis is a key
ingredient for the description and computation of many fundamental
invariants. The experiments shows that the letterplace methods
are computationally practicable and hence new noncommutative tasks
can be achieved now by commutative computer algebra systems.
Future research directions may consist in investigating relationships
between commutative and noncommutative invariants based on \Gr\ bases
and in developing optimized libraries for their computation.


\section*{Acknowledgments}

We gratefully acknowledge the support of the team of \textsf{Singular}
when performing the tests on their servers. We also like to thank Francesco
Brenti and Vesselin Drensky for suggesting most of the considered examples.
We finally express our gratitude to the anonymous reviewers for their valuable
remarks and suggestions.

\end{document}